\documentclass[11pt,oneside,a4paper]{article}

\usepackage{amsthm}
\usepackage{amssymb,amsmath}
\usepackage{subfig}
\usepackage[english]{babel}
\usepackage{graphicx}
\usepackage{epstopdf}
\usepackage{tabularx}
\usepackage{enumerate}

\makeatletter
\def\rightharpoonupfill@{%
  \arrowfill@\relbar\relbar\rightharpoonup}
\def\leftharpoondownfill@{%
  \arrowfill@\leftharpoondown\relbar\relbar}

\newcommand{\xrightleftharpoons}[2][]{\mathrel{%
\raise.22ex\hbox{%
$\ext@arrow 3095\rightharpoonupfill@{\phantom{#1}}{#2}$}%
\setbox0=\hbox{%
$\ext@arrow 0359\leftharpoondownfill@{#1}{\phantom{#2}}$}%
\kern-\wd0 \lower.22ex\box0}%
}
\makeatother

\makeatletter
\def\rightharrowfill@{%
  \arrowfill@\relbar\relbar\rightarrow}
\def\leftharrowfill@{%
  \arrowfill@\leftarrow\relbar\relbar}

\newcommand{\xrightleftharrow}[2][]{\mathrel{%
\raise.22ex\hbox{%
$\ext@arrow 3095\rightharrowfill@{\phantom{#1}}{#2}$}%
\setbox0=\hbox{%
$\ext@arrow 0359\leftharrowfill@{#1}{\phantom{#2}}$}%
\kern-\wd0 \lower.22ex\box0}%
}
\makeatother

\newtheorem{theorem}{Theorem}[section]

\newtheorem{lemma}[theorem]{Lemma}

\newtheorem{definition}[theorem]{Definition}

\newtheorem{remark}[theorem]{Remark}

\title{Dynamical aspects of the total QSSA in Enzyme Kinetics.}

\author{Alberto Maria Bersani$^1$  \and
        Enrico Bersani$^2$ \and
        Alessandro Borri$^3$ \and
        Pierluigi Vellucci$^4$ \\ pierluigi.vellucci@sbai.uniroma1.it 
}

\begin{document}
\maketitle

\begin{abstract}
In this paper we prove that the well-known quasi-steady state approximations, commonly used in enzyme kinetics, which can be interpreted as the reduced system of a differential system depending on a perturbative parameter, according to Tihonov theory, are asymptotically equivalent to the center manifold of the system. This allows to give a mathematical foundation for the application of a mechanistic method to determine the center manifold of (at this moment, still simple) enzyme reactions.

\end{abstract}

\section{Introduction}

The{\let\thefootnote\relax\footnotetext{Keywords: Michaelis-Menten kinetics, Center manifold, Tihonov's Theorem, quasi-steady state approximation, singular perturbation.}} mathematical{\let\thefootnote\relax\footnotetext{$^1$ Dipartimento di Ingegneria Meccanica e Aerospaziale, Via Eudossiana n. 18, 00184 Roma.}}
{\let\thefootnote\relax\footnotetext{$^2$ Laboratorio di Strutture e Materiali Intelligenti - Sapienza University Palazzo Caetani, via San Pasquale snc, 04012 - Cisterna di Latina (LT) Italy.}}
{\let\thefootnote\relax\footnotetext{$^3$ Istituto di Analisi dei Sistemi ed Informatica ``Antonio Ruberti'' (IASI-CNR) Piazzale A. Moro, 7 00185 Rome - Italy.}}
{\let\thefootnote\relax\footnotetext{$^4$Dept. of Economics, University of Roma TRE, via Silvio D'Amico 77, 00145 Rome, Italy.}}
 study of enzyme kinetics (or Michaelis-Menten kinetics) is mainly based on the standard Quasi-Steady State Approximation (sQSSA) - which has been used in Biochemistry, since the pioneering papers by Bodenstein \cite{bodenstein} and Chapman and Underhill \cite{chapman} - which starts from the observation that enzyme reactions are characterized by a first, short transient phase, where the intermediate complex rapidly grows, and a second, longer quasi-equilibrium phase, where the complex slowly decays in the product, in general the activated substrate. Each phase of the reaction has a time scale ($t_c$ and $t_s$, respectively).

The quasi-steady state approximation is a very efficient way of simplification for the description of a typical saturation phenomenon, occurring in enzyme kinetics, but present in several other biological systems. Let us cite, just as non exhaustive examples, the Monod-Wyman-Changeux molecular model of cooperativity in allosteric reactions \cite{monod}, or, more recently, the model of Lekszycki and coworkers concerning the bone regeneration \cite{Lu,Giorgio} and the model of infarcted cardiac tissue regeneration by means of stem cells \cite{stem}, where saturation phenomena are observed.

The reaction can be described as follows.

Let us consider an enzyme, $E$, which reacts with a protein, $X$, resulting in an intermediate complex, $C$. In turn, this complex can break down into a product, $X_p$, and the enzyme $E$. It is frequently assumed that the formation of $C$ is reversible while its breakup is not. The process is represented by the following sequence of reactions
\begin{equation}
\label{eq:a1}
X+E  \xrightleftharpoons[k_{-1}]{k_{1}}  C \xrightarrow{k_2} X_p+E
\end{equation}
where $k_1, k_{-1}, k_2$ are the reaction rates.

For notational convenience we will use variable names to denote both a chemical species and its concentration. For example, $E$ denotes both an enzyme and its concentration. Reaction (\ref{eq:a1}) obeys two natural constrains: the total amounts of protein and enzyme remain constant. Therefore,
\begin{equation}
\label{eq:a2}
X+C+X_p=X_T \ \ \text{and} \ \ E+C=E_T,
\end{equation}
for positive constants $X_T$ and $E_T$. In conjunction with the constraints (\ref{eq:a2}), the following Cauchy Problem for a system of two ordinary differential equations can be used to model reaction (\ref{eq:a1}):
\begin{align}
\label{eq:a3}
&\dot X(t)=-k_1 X(t) \left(E_T-C(t) \right)+k_{-1}C(t)\notag \\
 \dot C(t) & =k_1 X(t) \left(E_T-C(t) \right)-\left(k_{-1}+k_2\right)C(t) \notag \\
& = k_1 \left[ X(t) \left(E_T-C(t) \right)- K_M C(t) \right] \notag \\
& X(0)=X_T,   \quad C(0)=0.
\end{align}
where $\dot{x} (t) =\frac{d x}{dt}$ and where $E_T$, $S_T$, $k_1$, $k_2$, $k_{-1}$ are viewed as fixed positive constants and $K_M=(k_{-1}+k_2)/k_1$ is the Michaelis affinity constant. Similarly, we can define the dissociation constant $K_D =k_{-1}/k_1$ and the Van Slyke-Cullen constant $K=k_2/k_1$.

Since, after a short transient, where the complex $C$ rapidly grows, reaching its maximal concentration, it slowly decays, the sQSSA consists in supposing that, after the transient phase, the complex can be considered in a quasi-equilibrium, i.e., posing $\displaystyle \frac{dC}{dt} \cong 0$. With this approximation, the system becomes the  differential-algebraic system

\begin{align}
\label{eq:a3bis}
& C(t) = \frac{E_T X(t)}{X(t) + K_M},  \notag \\
&\dot X(t)=-k_2 C(t) = - \frac{V_{max} X(t)}{X(t) + K_M} ,\ \ \ \ \ \ \ \ \ \ X(0)=X_T ,
\end{align}
($V_{max} = k_2 E_T$) where only the initial condition $X(0)= X_T$ can be imposed, because the sQSSA describes only the slow phase, where the initial value of $C(t)$ is its maximal value, instead of $0$.

In the Sixties of the last Century mathematicians (see, in particular, \cite{Hein67}) interpreted the sQSSA in terms of leading order term of asymptotic expansions with respect to a perturbation parameter $\varepsilon$, which must be supposed small. Heineken et al. used $\varepsilon_{HTA} = E_T /S_T$, because in literature it is widely used to impose that the initial concentration of the enzyme $E$ is much less than the concentration of the substrate $X$.

The parameter can also arise by virtue of a biochemical condition imposing the separation between the two timescales $t_c$ and $t_s$ characterizing the reaction (see also \cite{lin,Segel1,Segel2,SS,palsson4}). In this way, Segel-Slemrod \cite{SS} showed that the sQSSA can be obtained also as the leading order of an asymptotic expansion in terms of $\varepsilon_{SS} = \frac{E_T}{S_T + K_M}$, enlarging the parameter range of validity of the sQSSA.

Inspired by the papers by Laidler \cite{laidler}, Swoboda \cite{swoboda,swoboda1}, and Schauer and Heinrich \cite{schauer}, Borghans et al. \cite{borghans} introduced a different approximation, called total Quasi-Steady State approximation (tQSSA) which uses the new variable $\overline{X} = X + C$, called total substrate.

Formally, introducing the "lumped" variable $\bar X:=X+C$, problem (\ref{eq:a3}) can be rewritten as
\begin{align}
\label{eq:a4}
&\dot{\bar{X}}(t)=-k_2 C(t), \notag \\
&\dot C(t)=k_1 \left[\bar{X}(t) E_T-\left(\bar X (t)+E_T+K_M\right)C(t)+C^2(t)\right], \notag \\
& \overline{X}(0)=X_T, \quad C(0)=0.
\end{align}

Also the tQSSA posits that $C$ equilibrates quickly compared to $\bar X$.

Imposing also in this case a quasi-steady state approximation ($\displaystyle \frac{dC}{dt} \cong 0$), we obtain

\begin{equation}
  \label{eq:tqssa_single}
  \dot{\bar{X}} \approx - k \, C_{-} (\bar X), \quad \bar X(0)=X_T,
\end{equation}
where
\begin{equation}
  \label{eq:cminus_single}
  C_-(\bar X)=\frac{(E_T+K_M+\bar X)-\sqrt{(E_T+K_M+\bar X)^2-4E_T\bar X}}{2}
\end{equation}
is the only biologically allowed solution of $\displaystyle \frac{dC}{dt}=0$.

Let us remark that since, thanks to the conservations laws, $P(t) = X_T - \bar X (t)$, the tQSSA can be viewed as the other side of the coin of Laidler's theory, though the approach followed in \cite{schauer,borghans} implicitly contains much more information about the reliability of the approximation, as shown in \cite{Be14}.

Also the tQSSA can be seen as the leading order term of an asymptotic expansion  in terms of a suitable parameter,
 $\displaystyle \varepsilon =\frac{E_T}{(K_M+E_T+X_T)^2}$, producing a new approximation, which is valid in a much wider parameter range. The parameter, introduced in \cite{borghans}, appears already in a paper by Palsson \cite{palsson4}, where the author determines sufficient conditions for the application of any Quasi-Steady State Approximation, based again on the time scale separation.
Taking into account that the perturbation parameter is always less than $1/4$, its introduction in terms of time scale separation appears much more natural than the previous parameters.
This result gives a theoretical mathematical foundation of the choice of the parameter in the tQSSA.
Moreover, several authors (see, for example, \cite{Segel1,Segel2}) study the transient phase of the reaction supposing that in this phase $X$ does not change substantially. This hypothesis is not realistic, while, using the total substrate $\bar{X}$, we observe that at time 0, we have $\dot{\bar{X}}(0)=0$, which addresses much more naturally the request of small changes of the total substrate in the initial time of the reaction.

In Figure \ref{figura1} we show the different efficiency of the two quasi-steady state approximations, when the parameters are stressed in such a way that the sQSSA is no more valid.

\begin{figure}[htp]
\centering
\includegraphics[width=0.5\textwidth]{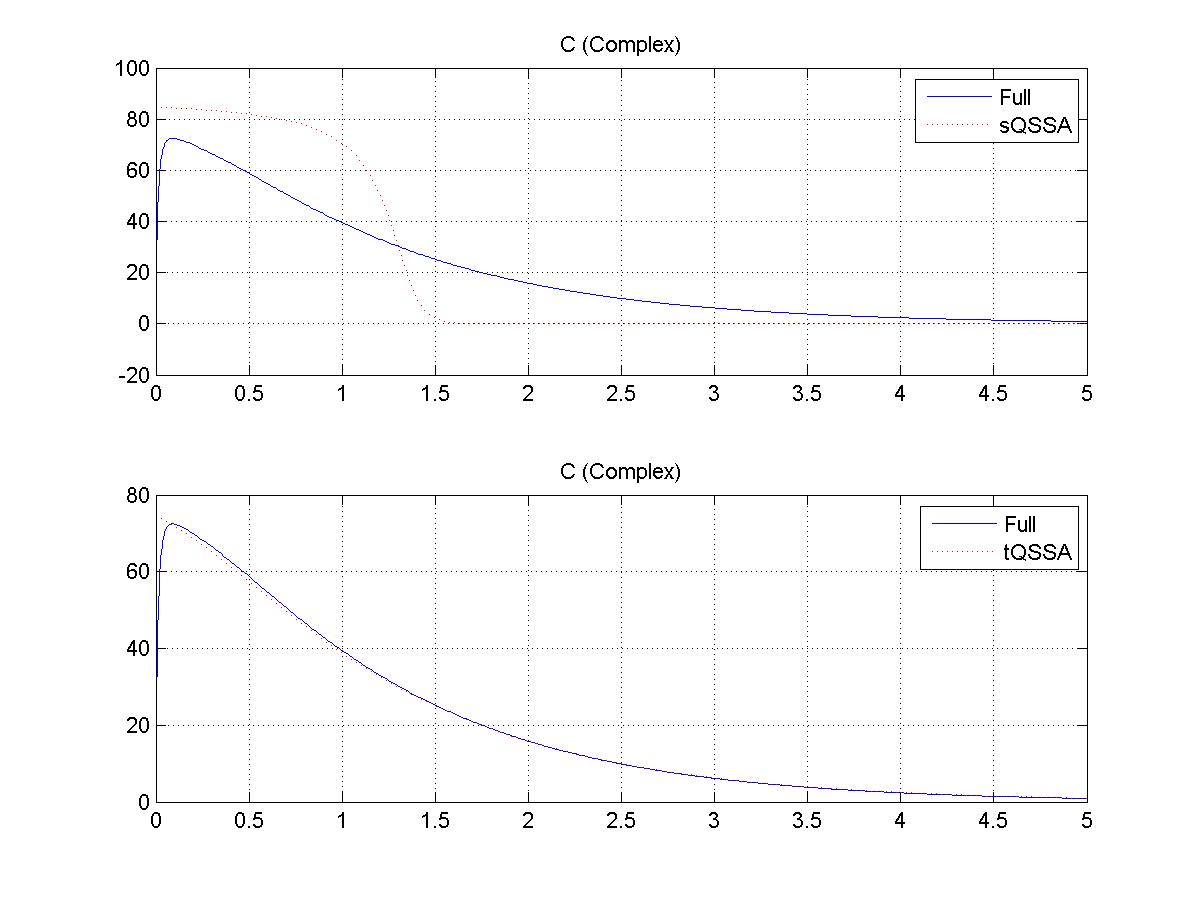}\hfil
\includegraphics[width=0.5\textwidth]{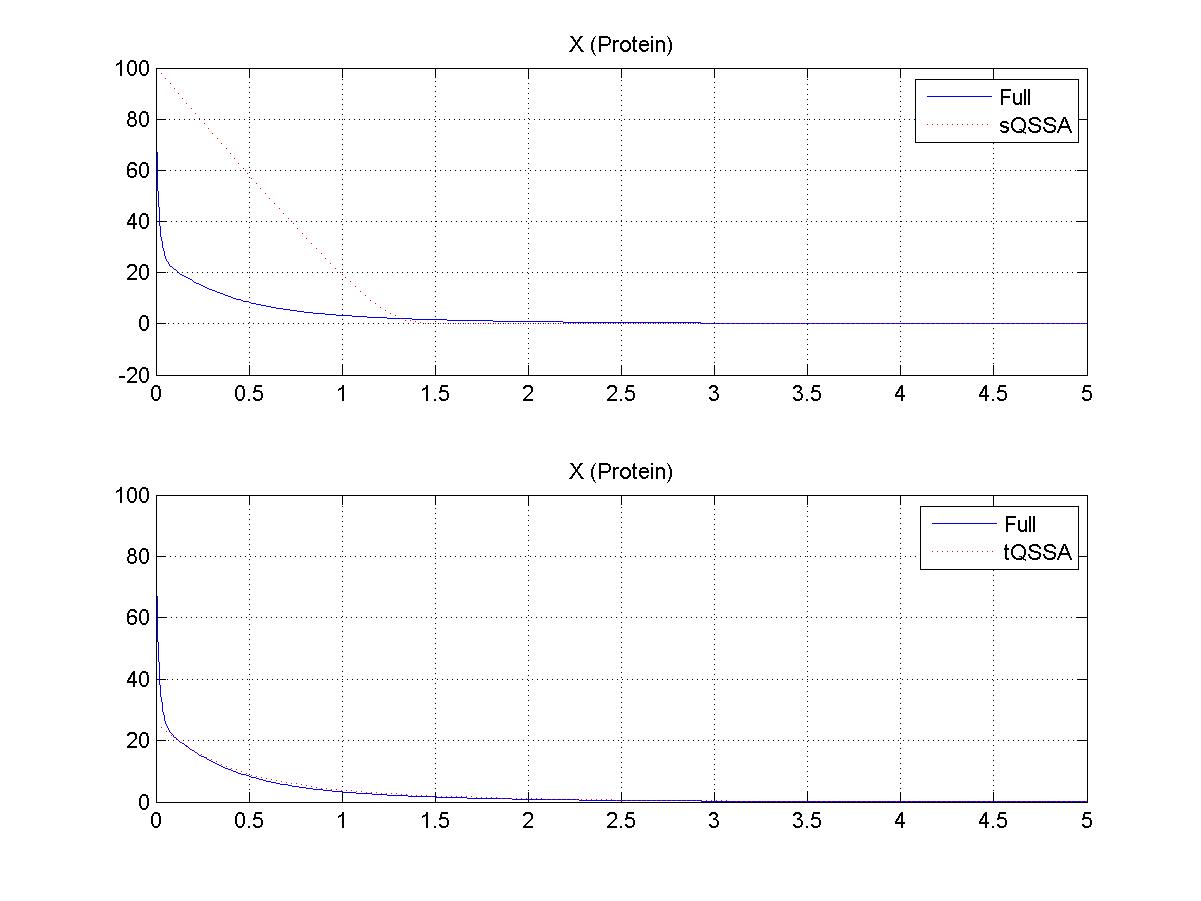}
\caption{Comparison of the complexes (left) and of the substrates (right), solution of the system (\ref{eq:a3}), with their sQSSA (\ref{eq:a3bis}) and tQSSA (\ref{eq:tqssa_single}). The parameter set is the following: $k_1=k_2=1; k_{-1}=4; E_T = 89; X_T = 100; K_M= 5; K= 4; \varepsilon_{SS} = 0.85; \varepsilon = 0.01$. The inadequacy of the sQSSA, mainly in the first part of the reaction, is evident, while the tQSSA in indistinguishable from the numerical solution of the system.}
\label{figura1}
\end{figure}

In previous literature the different QSSAs are approached by means of two different tools: Tihonov's Theorem \cite{Tik48,Tik50,Tik52}, which studies the asymptotic stability of systems of differential equations characterized by the presence of small perturbative parameters and Center Manifold Theory, which is one of the most powerful tools to study the dimensional reduction of differential systems. For example, on the one hand, Heineken et al. \cite{Hein67} and Dvo\v{r}\'{a}k and \v{S}i\v{s}ka \cite{dvorak} quote Tihonov's Theorem in order to justify the sQSSA, while Khho and Hegland \cite{Khoo} refer to this theorem to apply the tQSSA; on the other hand, other authors \cite{Nguyen,Ku11} interpret the sQSSA and the tQSSA, respectively, as the slow manifold of the Michaelis-Menten kinetics.

However, at the best of our knowledge, the two approaches are not yet compared, in order to check whether there exist any equivalences between the so-called singular points, introduced by Tihonov, and the center manifolds, as studied, for example, by Carr \cite{Carr81}.

Applying the techniques exposed in \cite{Wigg2,Wigg}, we show that the two approximations are asymptotically equivalent, concluding that the sQSSA and the tQSSA can be interpreted both as the singular point of the Michaelis-Menten kinetics and its center manifold.

This means that Tihonov's Theorem implies that any QSSA can be mathematically interpreted as the study of the reduced system of the original system setting the perturbative parameter $\varepsilon = 0$, instead of setting the derivative of the complex $C(t)$ equal to zero.

This fact formally allows the application of a "mechanistic" passage, consisting in equating to zero the derivatives of the complexes, in the single reaction scheme, as in more complex reactions, because this is the simplest way to reach (an approximate) expression of the center manifold. For example, Wang and Sontag \cite{wang2} apply this technique for the study of the sQSSA of the double phosphorylation-dephosphorylation cycle.

As shown in \cite{us_febs}, however, these approximation are no more applicable to mechanisms where oscillations can appear and an a priori analysis of the application should be performed every time we have to deal with any QSSA.

The paper is organized as follows in Section 2 we recall all the main definitions and properties of Tihonov's Theory and of the Center Manifold Theory; in Section 3 we show the equivalence of the two approaches in the case of the sQSSA, of the tQSSA and for a class of more general systems of differential equations characterized by the presence of a small, perturbative parameter; in Section 4 we discuss some future applications of this results to more complex enzymatic reactions.

\section{Preliminary results and notations on nonlinear dynamical systems.}

In this section we introduce the notations and summarize the results we need for the formulation of the problem investigated here. For convenience of the reader we closely follow the notations of the fundamental book by Wiggins \cite{Wigg}.

We will investigate systems in the class of general autonomous vector fields
\begin{equation}
\label{eq:1}
\dot x=f(x), \ \ \ \ x\in\mathbb R^n.
\end{equation}

It is natural to consider the linearized system
\begin{equation}
\label{eq:7}
\dot{y}=Ay, \ \ \ y\in\mathbb R^n,
\end{equation}
associated to the vector field (\ref{eq:1}), if $\bar x\in\mathbb R^n$ is one of its fixed points, and the constant Jacobian $n\times n$ matrix $A=Df(\bar x)$. Then $\mathbb R^n$ can be represented as the direct sum of three subspaces denoted $E^s$, $E^u$, and $E^c$, which are defined as follows:
\begin{align}
\label{eq:8}
&E^s=\operatorname{span}\{e_1,...,e_s\}, \notag \\
&E^u=\operatorname{span}\{e_{s+1},...,e_{s+u}\}, \ \ \ \ \ \ \  s+u+c=n \notag \\
&E^c=\operatorname{span}\{e_{s+u+1},...,e_{s+u+c}\},
\end{align}
where $\{e_1,...,e_s\}$, $\{e_{s+1},...,e_{s+u}\}$, $\{e_{s+u+1},...,e_{s+u+c}\}$ are the (generalized) eigenvectors of $A$ corresponding to the eigenvalues having negative real part, positive real part and zero real part, respectively. $E^s$, $E^u$, and $E^c$ are referred to as the stable, unstable, and center subspaces, respectively.
They are invariant subspaces (or manifolds) since solutions of (\ref{eq:7}) with initial conditions entirely contained in either $E^s$, $E^u$, or $E^c$ must forever remain in that particular subspace for all time.

It is well known that there exists a linear transformation $T$ which transforms the linear equation (\ref{eq:7}) into block diagonal form
\begin{equation}
\label{eq:14}
\left(\begin{array}{c}
          \dot{u} \\
          \dot{v} \\
          \dot w
        \end{array}
\right)=\left(
          \begin{array}{ccc}
            A_s & 0 & 0 \\
            0 & A_u & 0 \\
            0 & 0 & A_c \\
          \end{array}
        \right)
\left(\begin{array}{c}
          u \\
          v \\
          w
        \end{array}
\right),
\end{equation}
where $T^{-1}y\equiv(u,v,w)\in\mathbb R^s\times \mathbb R^u\times \mathbb R^c$, $s + u + c = n$, $A_s$ is an $s \times s$ matrix having eigenvalues with negative real part, $A_u$ is an $u \times u$ matrix having eigenvalues with positive real part, and $A_c$ is an $c\times c$ matrix having eigenvalues with zero real part. The $0$ in the block diagonal form (\ref{eq:14}) indicate appropriately sized block consisting of all zero's. Using this same linear transformation to transform the coordinates of the nonlinear vector field (\ref{eq:1}) gives the equation
\begin{equation}
\label{eq:13}
\begin{cases}
&\dot u=A_s u+R_s(u,v,w), \\
&\dot v=A_u v+R_u(u,v,w), \\
&\dot w=A_c w+R_c(u,v,w),
\end{cases}
\end{equation}
where $R_s(u, v,w)$, $R_u(u, v,w)$, and $R_c(u, v,w)$ are the first $s$, $u$, and $c$ components, respectively, of the vector $T^{-1}R Ty$.

The following theorem shows how this structure changes when the nonlinear vector field (\ref{eq:13}) is considered. It is stated without proof (see \cite{Wigg2} for details).

\begin{theorem}[Local Stable, Unstable, and Center Manifolds of Fixed Points]
\label{th:2}
Suppose (\ref{eq:13}) is $C^r$, $r \geq 2$. Then the fixed point $(u, v,w) = 0$ of (\ref{eq:13}) possesses a $C^r$ s-dimensional local, invariant stable manifold, $W^s_{loc}(0)$, a $C^r$ u-dimensional local, invariant unstable manifold, $W^u_{loc}(0)$, and a $C^r$ c-dimensional local, invariant center manifold, $W^c_{loc}(0)$, all intersecting in $(u, v,w) = 0$. These manifolds are all tangent to the respective invariant subspaces of the linear vector field (\ref{eq:14}) at the origin and, hence, are locally representable as graphs. In particular, we have
$$W^s_{loc}(0)=\Bigl\{(u,v,w)\in\mathbb R^s\times\mathbb R^u\times\mathbb R^c|v=h^s_v(u), w=h^s_w(u);$$
$$Dh^s_v(0)=Dh^s_w(0)=0;\ |u| \ \text{sufficiently small} \Bigr\}$$

$$W^u_{loc}(0)=\Bigl\{(u,v,w)\in\mathbb R^s\times\mathbb R^u\times\mathbb R^c|u=h^u_u(v), w=h^u_w(v);$$
$$Dh^u_u(0)=Dh^u_w(0)=0;\ |v| \ \text{sufficiently small} \Bigr\}$$

$$W^c_{loc}(0)=\Bigl\{(u,v,w)\in\mathbb R^s\times\mathbb R^u\times\mathbb R^c|u=h^c_u(w), v=h^c_v(w);$$
$$Dh^c_u(0)=Dh^c_v(0)=0;\ |w| \ \text{sufficiently small} \Bigr\}$$
where $h^s_v(u)$, $h^s_w(u)$, $h^u_u(v)$, $h^u_w(v)$, $h^c_u(w)$, and $h^c_v(w)$ are $C^r$ functions. Moreover, trajectories in $W^s_{loc}(0)$ and $W^u_{loc}(0)$ have the same asymptotic properties as trajectories in $E^s$ and $E^u$, respectively. Namely, trajectories of (\ref{eq:13}) with initial conditions in $W^s_{loc}(0)$ (resp., $W^u_{loc}(0)$) approach the origin at an exponential rate asymptotically as $t \rightarrow +\infty$ (resp., $t \rightarrow -\infty$).
\end{theorem}

If the eigenvalues of the center subspace are all precisely zero - rather than having just real part zero - then a center manifold is called a \emph{slow manifold}.

\subsection{Center Manifolds} If $E^u=\emptyset$, then any orbit will rapidly decay to $E^c$. Thus, in order to investigate the long-time behavior (i.e., stability) we need only to investigate the system restricted to $E^c$. This simple reasoning is the foundation of the ``reduction principle'' applied to the study of the stability of nonhyperbolic fixed points of nonlinear vector fields.

For our purposes, let us consider vector fields of the following form
\begin{align}
\label{eq:34}
& \dot x= Ax+f(x,y),\notag \\
&\dot y=By+g(x,y), \ \ (x,y)\in\mathbb R^c\times \mathbb R^s,
\end{align}
where
\begin{align}
&f(0,0)=0, \ \ Df(0,0)=0, \notag \\
&g(0,0)=0, \ \ Dg(0,0)=0.
\end{align}
In the above, $A$ is a $c\times c$ matrix having eigenvalues with zero real parts, $B$ is an $s\times s$ matrix having eigenvalues with negative real parts, and $f$ and $g$ are $C^r$ functions ($r\geq 2$).

For the sake of notation simplicity, let us write
the center manifold in the following way:
\begin{equation}
\label{eq:center}
W^c(0)=\{(x,y)\in\mathbb R^c\times\mathbb R^s|y=h(x),|x|<\delta, h(0)=0,Dh(0)=0\}
\end{equation}
with $\delta$ sufficiently small.

\begin{remark}
\label{r:1}
We remark that the conditions $h(0) = 0$ and $Dh(0) = 0$ imply that $W^c(0)$ is tangent to $E^c$ at $(x, y) = (0, 0)$.
\end{remark}

The following three theorems are taken from the seminal book \cite{Carr81}, as reported in \cite{Wigg}.

\begin{theorem}[Existence]
\label{th:3}
There exists a $C^r$ center manifold for (\ref{eq:34}). The dynamics of (\ref{eq:34}) restricted to the center manifold is, for $u$ sufficiently small, given by the following c-dimensional vector field
\begin{equation}
\label{eq:35}
\dot u=Au+f(u,h(u)), \ \ u\in\mathbb R^c.
\end{equation}
\end{theorem}
The next result implies that the dynamics of (\ref{eq:35}) near $u = 0$ determines the dynamics of (\ref{eq:34}) near $(x, y) = (0, 0)$.
\begin{theorem}[Stability]
\label{th:4}
i) Let the zero solution of (\ref{eq:35}) be stable (asymptotically stable) (unstable); then the zero solution of (\ref{eq:34}) is also stable (asymptotically stable) (unstable).
ii) Let the zero solution of (\ref{eq:35}) be stable. Then if $(x(t),y(t))$ is a solution of (\ref{eq:34}) with $(x(0),y(0))$ sufficiently small, there is a solution $u(t)$ of (\ref{eq:35}) such that as $t\rightarrow\infty$
\begin{align}
&x(t) = u(t) + O\left(e^{-\gamma t}\right), \notag \\
&y(t) = h(u(t)) + O\left(e^{-\gamma t}\right),
\end{align}
where $\gamma>0$ is a constant.
\end{theorem}

It is possible to compute the center manifold so that we can reap the benefits of Theorem \ref{th:4}.
Using invariance of $W^c(0)$ under the dynamics of (\ref{eq:34}), we derive a quasi-linear partial differential equation that $h(x)$ must satisfy, in order for its graph to be a center manifold for (\ref{eq:34}). This is done as follows:
\begin{enumerate}
  \item The $(x, y)$ coordinates of any point on $W^c(0)$ must satisfy
  \begin{equation}
  \label{eq:36}
  y=h(x).
  \end{equation}
  \item Differentiating (\ref{eq:36}) with respect to time implies that the $(\dot x , \dot y)$ coordinates of any point on $W^c(0)$ must satisfy
  \begin{equation}
  \label{eq:37}
  \dot y=Dh(x) \dot x.
  \end{equation}
  \item Any point on $W^c(0)$ obeys the dynamics generated by (\ref{eq:34}). Therefore, substituting
  \begin{align}
  \label{eq:38}
  & \dot x= Ax+f(x,h(x)),\notag \\
  &\dot y=Bh(x)+g(x,h(x)), \ \ (x,y)\in\mathbb R^c\times \mathbb R^s,
  \end{align}
  into (\ref{eq:37}) gives
  \begin{equation}
  \label{eq:39}
  Bh(x)+g(x,h(x))=Dh(x) \left(Ax+f(x,h(x))\right).
  \end{equation}
  or
  \begin{equation}
  \label{eq:40}
  \mathcal N\left(h(x)\right)\equiv Dh(x) \left(Ax+f(x,h(x))\right)-Bh(x)-g(x,h(x))=0
  \end{equation}
\end{enumerate}
Then, to find a center manifold, all we need to do is to solve (\ref{eq:40}).
\begin{theorem}[Approximation, \cite{Wigg}]
\label{th:5}
Let $\phi:\mathbb R^c\rightarrow\mathbb R^s$ be a $C^1$ mapping with $\phi(0)=D\phi(0)=0$ such that $\mathcal N\left(\phi(x)\right)=O\left(|x|^q\right)$ as $x\rightarrow0$ for some $q>1$. Then
$$\left|h(x)-\phi(x)\right|=O\left(|x|^q\right), \ \ \text{as} \ x\rightarrow0.$$
\end{theorem}
The theorem gives us a method for computing an approximate solution of (\ref{eq:40}) to any desired degree of accuracy. So, for this task, we will employ power series expansions (note that by Remark \ref{r:1} power series expansions start from second order).

\medskip

Suppose now that (\ref{eq:34}) depends on a vector of parameters $\varepsilon\in\mathbb R^p$:
\begin{align}
\label{eq:43}
& \dot x= Ax+f(x,y,\varepsilon),\notag \\
&\dot y=By+g(x,y,\varepsilon), \ \ (x,y,\varepsilon)\in\mathbb R^c\times \mathbb R^s\times \mathbb R^p,
\end{align}
where
\begin{align}
&f(0,0,0)=0, \ \ Df(0,0,0)=0, \notag \\
&g(0,0,0)=0, \ \ Dg(0,0,0)=0.
\end{align}
with the same assumptions as in (\ref{eq:34}).

Following Wiggins \cite{Wigg2,Wigg}, we will handle parameterized systems including the parameter $\varepsilon$ as a new dependent variable as follows
\begin{align}
\label{eq:44}
& \dot x= Ax+f(x,y,\varepsilon),\notag \\
& \dot \varepsilon=0, \notag \\
&\dot y=By+g(x,y,\varepsilon), \ \ (x,y,\varepsilon)\in\mathbb R^c\times \mathbb R^s\times \mathbb R^p,
\end{align}
This system has a fixed point at $(x,y,\varepsilon)=(0,0,0)$. The matrix associated with the linearization of (\ref{eq:44}) around this fixed point has $s$ eigenvalues with negative real part and $c + p$ eigenvalues with zero real part. Let us now apply center manifold theory. Modifying definition given in formula (\ref{eq:center}), a center manifold will be represented as the graph of $h(x, \varepsilon)$ for $x$ and $\varepsilon$ sufficiently small. Theorem \ref{th:3} still applies, with the vector field reduced to the center manifold given by
\begin{align}
\label{eq:45}
& \dot u= Au+f\left(u,h(u,\varepsilon),\varepsilon\right),\notag \\
& \dot \varepsilon=0,\ \ \ \ \ \ \ \ \ \ \ \ \ \ \ \ \ \ \ \ \ \ \ \ \ \  \ (u,\varepsilon)\in\mathbb R^c\times \mathbb R^p.
\end{align}
Let us calculate the center manifold.
 Using invariance of the graph of $h(x,\varepsilon)$ under the dynamics generated by (\ref{eq:44}), we have
\begin{equation}
\label{eq:46}
\dot y=D_x h(x,\varepsilon) \dot x+D_\varepsilon h(x,\varepsilon) \dot\varepsilon=Bh(x,\varepsilon)+g(x,h(x,\varepsilon),\varepsilon).
\end{equation}

Substituting (\ref{eq:44}) into (\ref{eq:46}) results in the following quasi-linear partial differential equation that $h(x,\varepsilon)$ must satisfy in order for its graph to be a center manifold.
\begin{align}
\label{eq:48}
&\mathcal N\left(h(x,\varepsilon)\right)\equiv D_x h(x,\varepsilon) \left[Ax+f(x,h(x,\varepsilon),\varepsilon)\right]+\notag \\
&-Bh(x,\varepsilon)-g\left(x,h(x,\varepsilon),\varepsilon\right)=0
\end{align}

Although center manifolds exist, they do not need to be unique. This can be seen from a well-known example due to Anosov (see \cite{Si85}, \cite{Wigg}). It can be proven (see, among others, \cite{Carr81} as reported in \cite{Wigg}) that any couple of center manifolds of a given fixed point differ by (at most) exponentially small terms. Thus, the Taylor series expansions of any two center manifolds agree to all orders.

Moreover, it can be shown that, due to the attractive nature of the center manifold, certain orbits (for example, fixed points, periodic orbits, homoclinic orbits, and heteroclinic orbits) that remain close to the origin for all the time must be on every center manifold of a given fixed point.

\subsection{A different viewpoint: Singular Perturbations}
\label{sec:singpert}
For this subsection we will refer to the widespread book by W. Wasow \cite{Wa02}, and - in particular - to its relevant section on singular perturbations. A systematic study of the qualitative aspects of such singular perturbation problems can be found in a series of papers by Tihonov (\cite{Tik48}, \cite{Tik50} and \cite{Tik52}).

We consider differential systems of the form
\begin{align}
\label{eq:c1}
\frac{dx}{dt}&=f(x,y)\notag \\
\varepsilon\frac{dy}{dt}&=g(x,y),
\end{align}
where $x$ is $c$-dimensional vector and $y$ an $s$-dimensional vector. All variables are real, and $\varepsilon$ is positive.

We assume that:
\begin{itemize}
  \item [(A)] The functions $f$ and $g$ in (\ref{eq:c1}) are continuous in an open region $\Omega$ of the $(x,y)$-space.
  \item [(B)] There is an $s$-dimensional vector function $\phi(x)$
continuous in $\xi_1\leq x\leq \xi_2$ such that the points $(x,\phi(x))$, for all $\xi_1\leq x\leq \xi_2$, are in $\Omega$ and
$$g(x,\phi(x))\equiv0.$$
  \item [(C)] There exists a number $\eta>0$, independent of $x$, such that the relations
$$\|y-\phi(x)\|<\eta, \ \ y\neq\phi(x) \ \ \text{in} \ \ \xi_1\leq x\leq \xi_2$$
imply
$$g(x,y)\neq0, \ \ \text{in} \ \ \xi_1\leq x\leq \xi_2.$$
\end{itemize}
The function $\phi(x)$ will be referred to as a root of the equation $g(x,y) = 0$.
It is not excluded that $g(x,y) = 0$ may have other roots besides $\phi(x)$.
A root $\phi(x)$ that satisfies condition C will be called \emph{isolated} in $\xi_1\leq x\leq \xi_2$.
\begin{definition}
The system of differential equations
\begin{equation}
\label{eq:c2}
\varepsilon\frac{dy}{dt}=g(x,y)
\end{equation}
in which $x$ is a parameter, will be called the boundary layer equation belonging to the system (\ref{eq:c1}).

To (\ref{eq:c1}) there corresponds the reduced (or degenerate) system
\begin{align}
\label{eq:c1bis}
\frac{dx_0}{dt}&=f(x_0,y_0)\notag \\
0&=g(x_0,y_0).
\end{align}
The solutions of (\ref{eq:c1}) and (\ref{eq:c1bis}) define trajectories $\left(x(t,\varepsilon),y(t,\varepsilon)\right)$ and $\left(x_0(t),y_0(t)\right)$ in the $(x,y)$-space.

We also assume:
\end{definition}
\begin{itemize}
  \item [(D)] The singular point $y=\phi(x)$ of the boundary layer equation (\ref{eq:c2}) is asymptotically stable for all $\xi_1\leq x\leq \xi_2$.
\end{itemize}
The root $\phi(x)$ will be called, briefly, a stable root in $\xi_1\leq x\leq \xi_2$, if assumption (D) is satisfied.

In accordance with our previous terminology we refer to the problem consisting of equations (\ref{eq:c1}) together with the initial condition
\begin{equation}
\label{eq:c3}
x=\alpha, \ \ y=\beta, \ \ \text{for} \ t=0
\end{equation}
as the full problem. The reduced problem is here defined by
\begin{align}
\label{eq:c4}
\frac{dx}{dt}&=f(x,\phi(x))\notag \\
y&=\phi(x),
\end{align}
\begin{equation}
\label{eq:c5}
x=\alpha, \ \ \ \text{for} \ t=0
\end{equation}
The differential equation (\ref{eq:c4}) is, of course, obtained by setting $\varepsilon=0$ in (\ref{eq:c1}) and determining the root $y = \phi(x)$ of the equation $g(x,y) = 0$. Moreover, we assume:
\begin{itemize}
  \item [(E)] The full problem, as well as the reduced one, has a unique solution in an interval $0\leq t\leq T$.
  \item [(F)] The asymptotic stability of the singular point $y = \phi(x)$ is uniform with respect to $x$ in $\xi_1\leq x\leq \xi_2$.
\end{itemize}
Let $\mu>0$. The set of points in the $(x,y)$-space for which the inequalities
$$\|y-\phi(x)\|<\mu, \ \ \ \xi_1\leq x\leq \xi_2$$
hold will be called a ``$\mu$-tube''. The set
$$\|y-\phi(x)\|=\mu, \ \ \ \xi_1\leq x\leq \xi_2$$
constitutes the ``lateral boundary'' of the $\mu$-tube.
\begin{lemma}
\label{l:1}
Suppose assumptions (A) to (F) are satisfied. Let $\mu>0$ be arbitrary but so small that the closure of the $\mu$-tube lies in $\Omega$. There exist then two numbers $\gamma(\mu)$ and $\varepsilon(\mu)$ such that for $\varepsilon<\varepsilon(\mu)$ the following is true: Any solution of the full equation that is in the interior of the $\mu$-tube for some value $\tilde{t}$ of $t$, $0\leq \tilde{t}\leq T$, and in the closure of the $\mu$-tube for all $t$ in $\tilde{t} \leq t < T$, does not meet the lateral surface of the $\mu$-tube for $\tilde{t} \leq t < T$.
\end{lemma}
The lemma states that, for small $\varepsilon$, any solution that comes close to the curve $y = \phi(x)$ in $\xi_1\leq x\leq \xi_2$ remains close to it, as long as $\xi_1\leq x\leq \xi_2$.

For a convenient formulation of Tihonov's Theorem, according to \cite{Wa02}, we introduce one more term.
\begin{definition}
\label{def:9}
A point $(\alpha,\beta) \in \Omega$, $\xi_1\leq \alpha\leq \xi_2$ is said to lie in the domain of influence of the stable root $y = \phi(x)$ if the solution of the problem
$$dy/d\tau = g(\alpha,y),\ \  y(0) = \beta$$
exists and remains in $\Omega$ for all $\tau > 0$, and if it tends to $\phi(\alpha)$, as $\tau\rightarrow+\infty$.
\end{definition}
\begin{theorem}
\label{th:6}
Let Assumptions (A) to (F) be satisfied and let $(\alpha,\beta)$ be a point in the domain of influence of the root $y = \phi(x)$. Then the solution $x(t,\varepsilon)$, $y(t,\varepsilon)$ of the full initial value problem (\ref{eq:c1}), (\ref{eq:c3}) is linked with the solution $(x_0(t)$, $y_0(t) = \phi(x_0(t)))$ of the reduced problem (\ref{eq:c4}), (\ref{eq:c5}) by the limiting relations
\begin{align}
\lim_{\varepsilon\rightarrow0}x(t,\varepsilon)&=x_0(t), \qquad \qquad \qquad 0\leq t\leq T_0\notag \\
\lim_{\varepsilon\rightarrow0}y(t,\varepsilon)&=y_0(t)=\phi(x_0(t)) \qquad  0< t\leq T_0
\end{align}
Here $T_0$ is any number such that $y = \phi(x_0(t))$ is an isolated stable root of $g\left(x_0(t),y\right) = 0$ for $0\leq t\leq T_0$. The convergence is uniform in $0\leq t\leq T_0$,
for $x(t,\varepsilon)$, and in any interval $0<t_1\leq t\leq T_0$ for $y(t,\varepsilon)$.
\end{theorem}
Tihonov's Theorem \ref{th:6} is only the first step in the asymptotic solution of initial value problems of the singular perturbation type. The most natural approach to this problems is to attempt a solution ({\it outer solution}) in the form of a series in powers of $\varepsilon$:
\begin{equation}
\label{eq:c6}
x=\sum_{r=0}^\infty x_r(t) \varepsilon^r, \ \ \ y=\sum_{r=0}^\infty y_r(t) \varepsilon^r
\end{equation}
and to determine the coefficients $x_r(t)$, $y_r(t)$ by means of formal substitution and comparison of coefficients.

It is clear that we have to relate the series (\ref{eq:c6}) to the behavior of the solution of (\ref{eq:c1}) in the boundary layer, as shown, for example, in \cite{Hein67,SS,perturbation}. For values of $t$ that are of order $O(\varepsilon)$ the solution to our perturbation problem can be found starting from the stretching transformation $t=\tau \varepsilon$. Hence, the stretched form of the original problem is
\begin{align}
\label{eq:c11}
\frac{dx}{d\tau}&=\varepsilon f\left(x,y\right), \ \ \ \frac{dy}{d\tau}=g\left(x,y\right), \notag \\
x&=\alpha, \ \ \ y=\beta, \ \ \ \text{for} \ \tau=0.
\end{align}
Also in this case we determine the solution of the transient phase ({\it inner solution}) in terms of a series in powers of $\varepsilon$.

The developments of the passages is beyond the scope of this paper. For the other accounts and a more detailed discussion, see \cite{Wa02}.

\section{Main results}

\subsection{The Heineken-Tsuchiya-Aris system}

Let us consider the enzymatic reaction described in (\ref{eq:a1}) and (\ref{eq:a3}).
Clearly $(X,C)=(0,0)$ is a fixed point of (\ref{eq:a3}).
Following \cite{Hein67}, let us first adimensionalize equations (\ref{eq:a3}). Let us observe that we could use different adimensionalization procedures, in particular using  the parameter $\varepsilon_{SS} =\frac{E_T}{S_T + K_M}$, as proposed in  \cite{SS}. However, we follow the simpler scheme shown in \cite{Hein67}, just in order to test our theoretical results and compare them with the results shown in \cite{Carr81}:
\begin{align}
\label{eq:a9s}
\frac{d u}{d\tau}&=-u + (u + \kappa - \lambda)v,\ \ \ \ \ \ \ u(0)=1, \notag \\
\varepsilon\frac{d v}{d\tau}&=u - (u + \kappa)v, \ \ \ \ \ \ \ \ \ \ \ \ \ \ \ v(0)=0.
\end{align}
where
$$\tau=k_1 E_T t, \ \ u=\frac{X}{X_T}, \ \ v=\frac{C }{E_T}, \ \ \lambda=\frac{k_2}{k_1 X_T},$$
and
$$\kappa=\frac{k_2+k_{-1}}{k_1 X_T}=\frac{K_M}{X_T}, \ \ \varepsilon=\frac{E_T}{X_T}.$$
This is the Heineken-Tsuchiya-Aris system \cite{Hein67}. Carr (\cite{Carr81}, p.8, example 3) uses equations
\begin{align}
\label{eq:a10carr}
\frac{d u}{d\tau}&=-u + (u + c)v,\notag \\
\varepsilon\frac{d v}{d\tau}&=u - (u + 1)v.
\end{align}
To obtain (\ref{eq:a9s}) from (\ref{eq:a10carr}) just impose $\kappa+\lambda=c$ and $\kappa=1$. We will start from (\ref{eq:a9s}) for having a more realistic system. By applying the sQSSA (which corresponds to impose $\varepsilon=0$), we have the reduced system ({\it outer solution}) of (\ref{eq:a9s})
\begin{align}
\label{eq:a11}
\frac{d u}{d\tau}&=-\frac{\lambda u}{\kappa+u},\notag \\
v&=\frac{u}{\kappa+u}.
\end{align}

As above remarked, Heineken et al. \cite{Hein67} and Dvo\v{r}\'{a}k and \v{S}i\v{s}ka \cite{dvorak} quote Tihonov's Theorem in order to justify the sQSSA.

The aim of this subsection is now to determine the center manifold for (\ref{eq:a9s}), using the techniques described in \cite{Wigg2,Wigg} and to show that it is asymptotically equivalent to the singular points related to Tihonov theory.

To this aim, let us now set $\tau=\varepsilon s$. Equations (\ref{eq:a9s}) can be rewritten in the equivalent form ({\it inner solution})
\begin{align}
\label{eq:a13}
\frac{d u}{ds}&=\varepsilon \varphi(u,v),  \notag \\
\frac{d v}{ds}&=u - (u + \kappa)v.
\end{align}
where $\varphi(u,v)=-u + (u + \kappa - \lambda)v$. In order to obtain for (\ref{eq:a13}) a block form, of the type (\ref{eq:34}), where the submatrix having eigenvalues with zero real parts is separated from the submatrix having eigenvalues with negative real parts, we operate the substitution $w:=u-\kappa v$, i.e., $v=\frac{u-w}{\kappa}$. Hence,
$$\varphi(u,w)=-u + (u + \kappa - \lambda)\frac{u-w}{\kappa}$$
and
\begin{align}
\frac{d w}{ds}&=\frac{d u}{ds}-\kappa \frac{d v}{ds}  \notag \\
&=\varepsilon \varphi(u,w)-\kappa u+(u + \kappa)(u-w)=\varepsilon \varphi(u,w)+u(u-w)-\kappa w.
\end{align}
Following \cite{Wigg}, the way we will handle parametrized systems consists of including the parameter $\varepsilon$ as a new dependent variable as in (\ref{eq:a14}), which merely acts to augment the matrix $A$ by adding a new center direction that has no dynamics. In this way, system (\ref{eq:a13}) becomes
\begin{align}
\label{eq:a14}
\frac{d u}{ds}&=\varepsilon \varphi(u,w),  \notag \\
\frac{d w}{ds}&=-\kappa w+u(u-w)+\varepsilon \varphi(u,w), \notag \\
\frac{d \varepsilon}{ds}&=0
\end{align}
where the parameter $\varepsilon$ is a new variable and the system could have also other fixed points.

The associated linearized system has a diagonal form, where the eigenvalues are given by $0$ (multiplicity $2$) and $-\kappa$.

To find a center manifold, all we need to do is to solve (\ref{eq:40}) for system (\ref{eq:a14}), employing Theorem \ref{th:5}, which gives us a method for computing an approximate solution of (\ref{eq:40}) to any desired degree of accuracy. Referring to (\ref{eq:40}) and (\ref{eq:34}), $A=0$, $B=-\kappa$, so we search for a function $w = h(u,\varepsilon)$ such that
\begin{equation}
\label{eq:a15}
D_u h(u,\varepsilon) \left(0+f(u,h(u,\varepsilon),\varepsilon)\right)+\kappa h(u,\varepsilon)-g(u,h(u,\varepsilon),\varepsilon)=0
\end{equation}
where
\begin{align}
f(u,h(u,\varepsilon),\varepsilon)&=\varepsilon \varphi\left(u,h(u,\varepsilon)\right),  \notag \\
g(u,h(u,\varepsilon),\varepsilon)&=u\left(u-h(u,\varepsilon)\right)+\varepsilon \varphi\left(u,h(u,\varepsilon)\right).
\end{align}
Using Theorem \ref{th:5} we assume
\begin{equation}
\label{eq:a16}
h(u,\varepsilon)=a_1 u^2+a_2 u \varepsilon+ a_3 \varepsilon^2+\dots
\end{equation}
Substituting (\ref{eq:a16}) into (\ref{eq:a15}), one has:
\begin{align}
\label{eq:a17}
\varepsilon&\left(2a_1 u+a_2 \varepsilon+\dots\right) \varphi\left(u,h(u,\varepsilon)\right)+ \kappa \left(a_1 u^2+a_2 u \varepsilon+\dots\right)+  \notag \\
-&u\left(u-a_1 u^2-a_2 u \varepsilon +\dots\right)-\varepsilon \varphi\left(u,h(u,\varepsilon)\right)=0
\end{align}
where
\begin{align}
\label{eq:a18}
\varphi\left(u,h(u,\varepsilon)\right)&=-a_1 u^2-a_2 u \varepsilon+\dots -\frac{\lambda}{\kappa} \left(u-a_1 u^2-a_2 u \varepsilon+ \dots\right) \notag \\
&+\frac{u}{\kappa} \left(u-a_1 u^2-a_2 u \varepsilon+ \dots\right) \notag \\
&=-\frac{\lambda}{\kappa} u + \left(-a_1+\frac{1}{\kappa}+\frac{\lambda}{\kappa}a_1\right) u^2+\left(-a_2+\frac{\lambda}{\kappa}a_2\right)u\varepsilon \notag \\
&+\left(-a_3+\frac{\lambda}{\kappa}a_3\right)\varepsilon^2+\dots
\end{align}
Accordingly, substituting (\ref{eq:a18}) into (\ref{eq:a17}), one has
\begin{align}
\label{eq:a19}
&\varepsilon\left(2a_1 u+a_2 \varepsilon+\dots\right)\Biggl[-\frac{\lambda}{\kappa} u + \left(-a_1+\frac{1}{\kappa}+\frac{\lambda}{\kappa}a_1\right) u^2+\left(-a_2+\frac{\lambda}{\kappa}a_2\right)u\varepsilon+ \notag \\
&+ \left(-a_3+\frac{\lambda}{\kappa}a_3\right)\varepsilon^2+\dots \Biggr]+\kappa \left(a_1 u^2+a_2 u \varepsilon+a_3 \varepsilon^2+\dots\right)+   \notag \\
&-u\left(u-a_1 u^2-a_2 u \varepsilon-a_3 \varepsilon^2 +\dots\right)-\varepsilon \Biggl[-\frac{\lambda}{\kappa} u + \left(-a_1+\frac{1}{\kappa}+\frac{\lambda}{\kappa}a_1\right) u^2+\notag \\
&+\left(-a_2+\frac{\lambda}{\kappa}a_2\right)u\varepsilon+ \left(-a_3+\frac{\lambda}{\kappa}a_3\right)\varepsilon^2+\dots\Biggr] =0
\end{align}
Truncating at second order terms, we obtain:
$$\left(k a_1-1\right)u^2+\left(k a_2+\frac{\lambda}{\kappa}\right)u\varepsilon+\kappa a_3 \varepsilon^2+\dots=0$$
Equating terms of the same power to zero gives $a_1=\frac{1}{\kappa}$, $a_2=-\frac{\lambda}{\kappa^2}$ and $a_3=0$. Hence, the center manifold for system (\ref{eq:a14}) is:
\begin{equation}
\label{eq:a20}
h(u,\varepsilon)=\frac{1}{\kappa} u^2-\frac{\lambda}{\kappa^2} u \varepsilon+\dots
\end{equation}
which, for $\kappa=1$, gives the result shown in \cite{Carr81}. Finally, substituting (\ref{eq:a20}) into (\ref{eq:a14}), we obtain the vector field reduced to the center manifold, according to equation (\ref{eq:35}) of Theorem \ref{th:3}. In fact, if $a_1=\frac{1}{\kappa}$, $a_2=-\frac{\lambda}{\kappa^2}$ and $a_3=0$, formula (\ref{eq:a18}) becomes:
$$\varphi\left(u,h(u,\varepsilon)\right)=-\frac{\lambda}{\kappa} u +\frac{\lambda}{\kappa^2} u^2-\frac{\lambda}{\kappa^2} \left(-1+\frac{\lambda}{\kappa}\right)u\varepsilon+\dots$$
Thus:
\begin{align}
\label{eq:a21}
\frac{d u}{ds}&=\varepsilon \left[-\frac{\lambda}{\kappa} u +\frac{\lambda}{\kappa^2}  u^2-\frac{\lambda}{\kappa^2} \left(-1+\frac{\lambda}{\kappa}\right)u\varepsilon+\dots\right],  \notag \\
\frac{d \varepsilon}{ds}&=0
\end{align}
or, in terms of the original time scale,
\begin{align}
\label{eq:a21bis}
\dot u&=\frac{\lambda}{\kappa} u \left[-1 +\frac{u}{\kappa}  -\frac{\varepsilon}{\kappa} \left(-1+\frac{\lambda}{\kappa}\right)+\dots\right],  \notag \\
\dot\varepsilon&=0
\end{align}

Let us now conclude showing that the center manifold obtained following this method is asymptotically sufficiently close to (\ref{eq:a11}). We can obtain back the equation in $v$. In fact, since $v=\frac{u-w}{\kappa}$, from (\ref{eq:a20}) and
$$w=h(u,\varepsilon)=\frac{1}{\kappa} u^2-\frac{\lambda}{\kappa^2} u \varepsilon+\dots$$
we have
\begin{equation}
\label{eq:aaa20}
v=\frac{u}{\kappa}\left(1-\frac{w}{u}\right)=\frac{u}{\kappa}\left(1-\frac{1}{\kappa} u+\frac{\lambda}{\kappa^2}\varepsilon+\dots\right)
\end{equation}
Considering $\varepsilon \ll 1$, one has
\begin{equation}
\label{eq:b10}
v\sim\frac{u}{\kappa}\left(1-\frac{u}{\kappa}\right)\sim \frac{u}{\kappa}\left(\frac{1}{1+\frac{u}{\kappa}}\right) = \frac{u}{\kappa + u} , \ \ \ \text{for} \ u\rightarrow0
\end{equation}
which is the second equation of (\ref{eq:a11}). We can conclude that, supposing $\varepsilon\ll 1$, the center manifold determined in this way approximates the solution given by the sQSSA, which coincides with the root related to the application of Tihonov's Theorem.

In Figure (\ref{figura2}) we compare the sQSSA of system (\ref{eq:a9s}), obtained from (\ref{eq:a10carr}), with the center manifold (\ref{eq:aaa20}), at the zeroth order and at the first order in $\varepsilon$, respectively. Obviously, the latter gives a better approximation of the numerical solution of (\ref{eq:a9s}), while the former well approximates the sQSSA curve, which in fact can be considered the zeroth order term of an asymptotic expansion of the solution in terms of $\varepsilon$.

\begin{figure}[htp]
\centering
\includegraphics[width=0.5\textwidth]{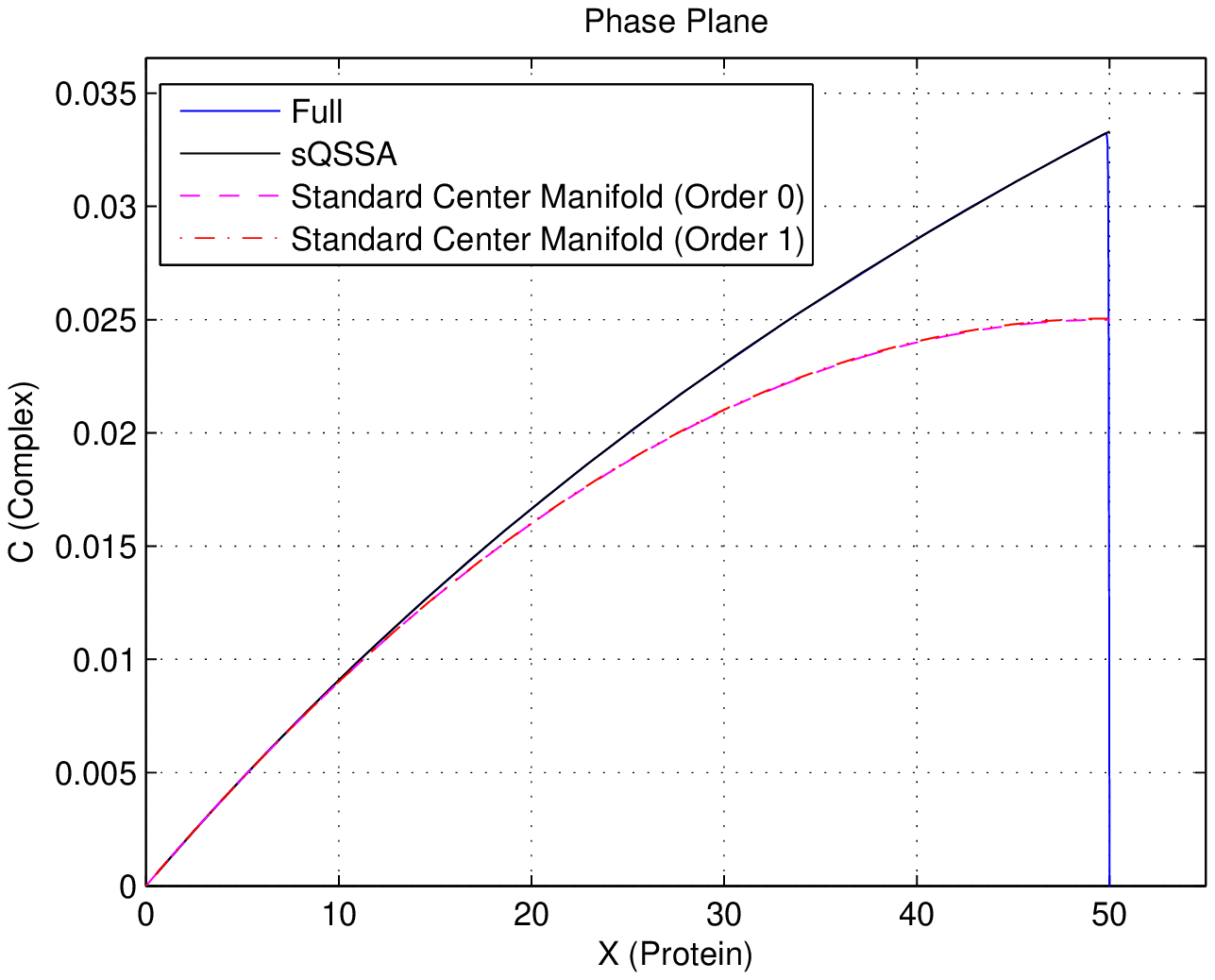}\hfil
\includegraphics[width=0.5\textwidth]{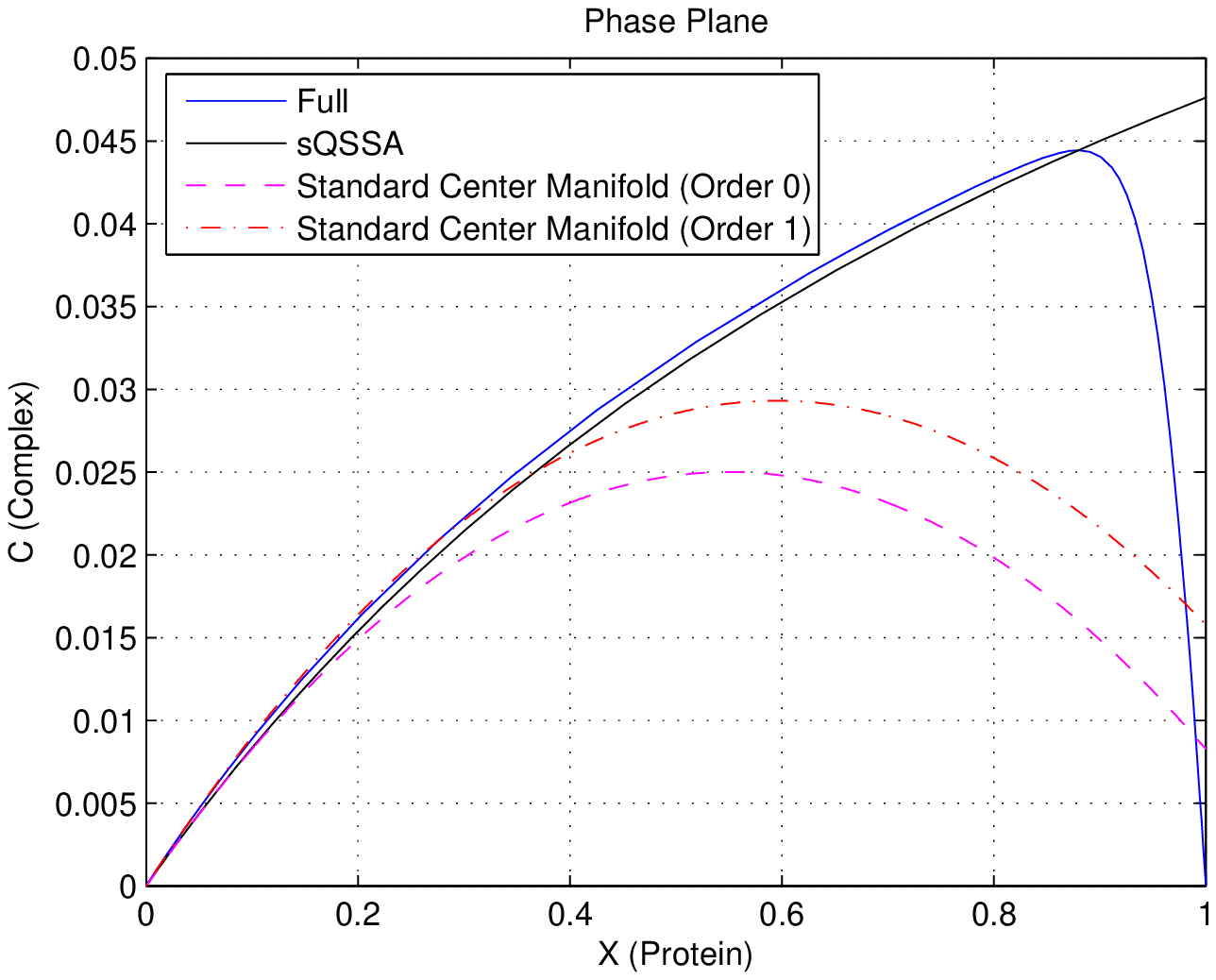}
\caption{
Comparison in the phase space $(X,C)$ of the numerical solution of the system (\ref{eq:a9s}) (blue solid line) with its sQSSA (\ref{eq:a10carr}) (black solid line) and its zeroth order (dashed line) and first order (dashed/dotted line) center manifold (\ref{eq:aaa20}). The parameter sets are the following.
Left: $k_1=0.1; k_2=10; k_{-1}=0.01; E_T=0.1; X_T=50; K_M=100.1; K=100; \varepsilon_{HTA}=0.002; \varepsilon_{SS}=0.0007$. The set was taken from \cite{bulgari}. Right: $k_1=k_2=1; k_{-1}=0.1; E_T=0.1; X_T=1; K_M=1.1; K=1; \varepsilon_{HTA}=0.1; \varepsilon_{SS}=0.05$. The set was taken (and modified) from \cite{Ku11}.
Since in both cases the value of $\varepsilon_{SS}$ is sufficiently small, the different approximations converge to the graph of the numerical solution. In the plot on the right it is possible to appreciate the different behavior of the zero-th order and the first order center manifolds. While the first order manifold approximates in a better way the numerical solution, the zero-th order converges to the sQSSA, that does not approximate sufficiently well the numerical solution, since it is the zero-th order term of the singular perturbation of the solution in terms of the parameter $\varepsilon_{SS}$.}
\label{figura2}
\end{figure}

\subsection{The total quasi-steady state approximation (tQSSA)}

Let us now consider system (\ref{eq:a4}). Let us adimensionalize the system, as in \cite{SchMai02,perturbation}
\begin{align}
\label{eq:aaa12}
\frac{d u}{d\tau}&=-v, \ \ \ \ \ \ \ \ \ \ \ \ \ \ \ \ \ \ \ \ \ \ \ \ \ \ \ \ \ \ \ \ \ \ \ \ \ \ \ \ \ \ \ u(0)=1, \notag \\
\varepsilon\frac{d v}{d\tau}&=\eta \sigma v^2 - (\eta + \kappa_m)v-\sigma u v+u, \ \ \ \ \ \ \ \ \ \ v(0)=0.
\end{align}
where
$$\bar{X}=u X_T, \ \ C=\left(\frac{E_T X_T}{E_T+K_M+X_T}\right) v, \ \ \tau=\frac{E_T+K_M+X_T}{k_2 E_T} t,$$
and
$$\varepsilon=\frac{K E_T}{\left(E_T+K_M+X_T\right)^2}, \ \ K=\frac{k_2}{k_1}$$
with system parameters
$$\sigma=\frac{X_T}{E_T+K_M+X_T}, \ \ \eta= \frac{E_T}{E_T+K_M+X_T}, \ \ \kappa_m=\frac{k_m}{E_T+K_M+X_T}$$
such that $\sigma+\eta+\kappa_m=1$.

By applying the tQSSA (which corresponds to imposing $\varepsilon=0$) to (\ref{eq:aaa12}), we have
$$\eta \sigma v^2 - (\eta + \kappa_m)v-\sigma u v+u=0$$
and, solving in $v$:
\begin{equation}
\label{eq:a29}
v=\frac{\eta + \kappa_m+\sigma u-\sqrt{\left(\eta + \kappa_m+\sigma u\right)^2-4\eta \sigma u}}{2\eta \sigma}
\end{equation}
which represents the singular point (or {\it outer solution}) of (\ref{eq:aaa12}), where $\eta$, $\sigma$, $\kappa_m$ are viewed as fixed positive constants and $\varepsilon$ is the parameter. Its fixed point is $(u,v)=(0,0)$.

The new parameter $\varepsilon$ appears already in \cite{palsson4} and was used in \cite{SchMai02,dingee,perturbation} to determine the asymptotic expansions whose leading term is just the tQSSA. In 2008 Khoo and Hegland \cite{Khoo} applied Tihonov's Theorem \cite{Tik48,Tik50} in order to study the tQSSA, obtaining similar results as in \cite{borghans}.

The aim of this subsection is now to determine the center manifold for (\ref{eq:aaa12}), using the techniques described in \cite{Wigg2,Wigg} and to show that they are asymptotically equivalent to the singular points related to Tihonov theory.

To this aim, let us now set $\tau=\varepsilon s$, system (\ref{eq:aaa12}) can be rewritten in the form ({\it inner solution})
\begin{equation}
\label{eq:a22}
\begin{cases}
\frac{d u}{ds}&=-\varepsilon v, \\
\frac{d v}{ds}&=u- (\eta + \kappa_m)v+\eta \sigma v^2 -\sigma u v \ .
\end{cases}
\end{equation}
In order to obtain for (\ref{eq:a22}) a block form, of the type (\ref{eq:34}), we make the substitution $w:=u- (\eta + \kappa_m)v$, i.e., $v=\frac{u-w}{\eta + \kappa_m}$. Hence,
$$\frac{d w}{ds}=\frac{d u}{ds}-(\eta + \kappa_m) \frac{d v}{ds}=\varepsilon \frac{w-u}{\eta + \kappa_m}+$$
$$- (\eta + \kappa_m)\left(u- (\eta + \kappa_m)\frac{u-w}{\eta + \kappa_m}+\eta \sigma \left(\frac{u-w}{\eta + \kappa_m}\right)^2 -\sigma u \frac{u-w}{\eta + \kappa_m}\right)$$
$$=- (\eta + \kappa_m) w+\sigma u (u-w)+\varepsilon \frac{w-u}{\eta + \kappa_m}- \eta \sigma \frac{\left(u-w\right)^2}{\eta + \kappa_m}$$
Doing so, and introducing the new variable $\varepsilon$, system (\ref{eq:a22}) becomes
\begin{equation}
\label{eq:a23}
\begin{cases}
\frac{d u}{ds}&=\varepsilon \frac{w-u}{\eta + \kappa_m} \\ \\
\frac{d w}{ds}&=- (\eta + \kappa_m) w+\sigma u (u-w)+\varepsilon \frac{w-u}{\eta + \kappa_m}- \eta \sigma \frac{\left(u-w\right)^2}{\eta + \kappa_m} \\
\dot\varepsilon&=0
\end{cases}
\end{equation}
The associated linearized system has a diagonal form and, in fact, the eigenvalues are given by $0$ (with multiplicity $2$) and $-(\eta + \kappa_m)$.

Also in this case the eigenvalue $0$ has multiplicity $2$.

We solve (\ref{eq:40}) for system (\ref{eq:a23}), employing Theorem \ref{th:5} and determine the center manifold. Referring to (\ref{eq:40}) and (\ref{eq:34}), we have that $A=0$, $B=-(\eta + \kappa_m)$. Accordingly, we search a function $h(u,\varepsilon)$ such that
\begin{align}
\label{eq:a25}
&D_u h(u,\varepsilon) \left(0+f(u,h(u,\varepsilon),\varepsilon)\right)+(\eta + \kappa_m) h(u,\varepsilon)-g(u,h(u,\varepsilon),\varepsilon)=0\notag \\
&f(u,h(u,\varepsilon),\varepsilon)=\varepsilon \frac{h(u,\varepsilon)-u}{\eta + \kappa_m},\notag \\
&g(u,h(u,\varepsilon),\varepsilon)=\sigma u (u-h(u,\varepsilon))+\varepsilon \frac{h(u,\varepsilon)-u}{\eta + \kappa_m}- \eta \sigma \frac{\left(u-h(u,\varepsilon)\right)^2}{\eta + \kappa_m}
\end{align}
Using Theorem \ref{th:5} we assume
\begin{equation}
\label{eq:a24}
h(u,\varepsilon)=a_1 u^2+a_2 u \varepsilon+ a_3 \varepsilon^2+\dots
\end{equation}
Substituting (\ref{eq:a24}) into (\ref{eq:a25}), one has:
\begin{align}
\label{eq:a26}
&\varepsilon\left(2a_1 u+a_2 \varepsilon+\dots\right)  \frac{-u+a_1 u^2+a_2 u \varepsilon+ a_3 \varepsilon^2+\dots}{\eta + \kappa_m}+\notag \\
&+ (\eta + \kappa_m) \left(a_1 u^2+a_2 u \varepsilon+ a_3 \varepsilon^2+\dots\right) +\sigma u (-u+a_1 u^2+a_2 u \varepsilon+ a_3 \varepsilon^2+\dots)+ \notag \\
&-\varepsilon \frac{-u+a_1 u^2+a_2 u \varepsilon+ a_3 \varepsilon^2+\dots}{\eta + \kappa_m}+\eta \sigma \frac{\left(-u+a_1 u^2+a_2 u \varepsilon+ a_3 \varepsilon^2+\dots\right)^2}{\eta + \kappa_m}=0
\end{align}
Truncating at second order terms, we obtain:
$$(\eta + \kappa_m) \left(a_1 u^2+a_2 u \varepsilon+ a_3 \varepsilon^2\right)-\sigma u^2+ \frac{\varepsilon u}{\eta + \kappa_m}+\eta \sigma \frac{u^2}{\eta + \kappa_m}=0$$
so,
$$ \left[(\eta + \kappa_m)a_1-\sigma+ \frac{\eta \sigma}{\eta + \kappa_m}\right] u^2+$$
$$+\left[(\eta + \kappa_m)a_2+\frac{1}{\eta + \kappa_m}\right] u \varepsilon+ (\eta + \kappa_m)a_3 \varepsilon^2 =0$$
from which:
$$a_1=\frac{\sigma}{\eta + \kappa_m}- \frac{\eta \sigma}{(\eta + \kappa_m)^2}, \ \ a_2=-\frac{1}{(\eta + \kappa_m)^2}, \ \ a_3=0$$
Hence, the center manifold for system (\ref{eq:a23}) is:
\begin{equation}
\label{eq:a27}
h(u,\varepsilon)=\frac{\sigma \kappa_m}{(\eta + \kappa_m)^2} u^2-\frac{1}{(\eta + \kappa_m)^2} u \varepsilon+\dots
\end{equation}
Finally, substituting (\ref{eq:a27}) into (\ref{eq:a23}) we obtain the vector field reduced to the center manifold, according to equation (\ref{eq:35}) of Theorem \ref{th:3}.
Then:
\begin{equation}
\label{eq:a28}
\begin{cases}
\frac{d u}{ds}&= \frac{\varepsilon}{\eta + \kappa_m}\left[-u+\frac{\sigma \kappa_m}{(\eta + \kappa_m)^2} u^2-\frac{1}{(\eta + \kappa_m)^2} u \varepsilon+\dots\right], \\
\dot\varepsilon&=0
\end{cases}
\end{equation}
or, in terms of the original time scale,
\begin{align}
\label{eq:a28bis}
\dot u&= \frac{u}{\eta + \kappa_m}\left[-1+\frac{\sigma \kappa_m}{(\eta + \kappa_m)^2} u-\frac{1}{(\eta + \kappa_m)^2} \varepsilon+\dots\right],  \notag \\
\dot\varepsilon&=0
\end{align}

Let us show that the center manifold obtained in (\ref{eq:a27}) is asymptotically close to the root given by (\ref{eq:a29}), in terms of Tihonov's Theorem. From (\ref{eq:a27}), and since $v=\frac{u-w}{\eta + \kappa_m}$, with
$$w=h(u,\varepsilon)=\frac{\sigma \kappa_m}{(\eta + \kappa_m)^2} u^2-\frac{1}{(\eta + \kappa_m)^2} u \varepsilon+\dots$$
we have

\begin{equation}
\label{eq:aaa30}
v=\frac{u}{\eta + \kappa_m}\left(1-\frac{w}{u}\right)=\frac{u}{\eta + \kappa_m}\left(1-\frac{\sigma \kappa_m}{(\eta + \kappa_m)^2} u+\frac{1}{(\eta + \kappa_m)^2} \varepsilon+\dots\right)
\end{equation}

Since equation (\ref{eq:a29}) is obtained putting $\varepsilon \ll 1$, one has
\begin{align}
\label{eq:a30}
v\sim\frac{u}{\eta + \kappa_m}\left(1-\frac{\sigma \kappa_m}{(\eta + \kappa_m)^2} u\right)&\sim \frac{u}{\eta + \kappa_m}\left(\frac{1}{1+\frac{\sigma \kappa_m}{(\eta + \kappa_m)^2}u}\right) \notag \\
&=\frac{u}{\eta + \kappa_m+\frac{\sigma \kappa_m}{\eta + \kappa_m} u}, \ \ \ \text{for} \ u\rightarrow0
\end{align}
while, approximating the square root in (\ref{eq:a29}), one has

\begin{align}
\label{eq:a31}
v & = \frac{\eta + \kappa_m+\sigma u}{2\eta \sigma}\left(1-\sqrt{1-\frac{4\eta \sigma u}{\left(\eta + \kappa_m+\sigma u\right)^2}}\right) =
\notag \\
& \frac{\eta + \kappa_m+\sigma u}{2\eta \sigma} \left[1- \sqrt{1 - 4 \varepsilon} \right] \sim \frac{\eta + \kappa_m+\sigma u}{\eta \sigma}
\varepsilon
= \frac{u}{\eta + \kappa_m+\sigma u}, \ \ \ \text{for} \ u\rightarrow0
\end{align}

It follows that both (\ref{eq:a30}) and (\ref{eq:a31}) are asymptotic to $\frac{u}{\eta + \kappa_m}$ when $u\rightarrow0$. This means that both the expressions can be intended as two different approximations of the center manifold.

In Figure (\ref{figura3}) we compare the tQSSA of system (\ref{eq:aaa12}), obtained from (\ref{eq:a29}), with the center manifold (\ref{eq:aaa30}), at the zeroth order and at the first order in $\varepsilon$, respectively. Obviously, the latter gives a better approximation of the numerical solution of (\ref{eq:aaa12}), while the former well approximates the tQSSA curve, which in fact can be considered the zeroth order term of an asymptotic expansion of the solution in terms of $\varepsilon$.

\begin{figure}[htp]
\centering
\includegraphics[width=0.5\textwidth]{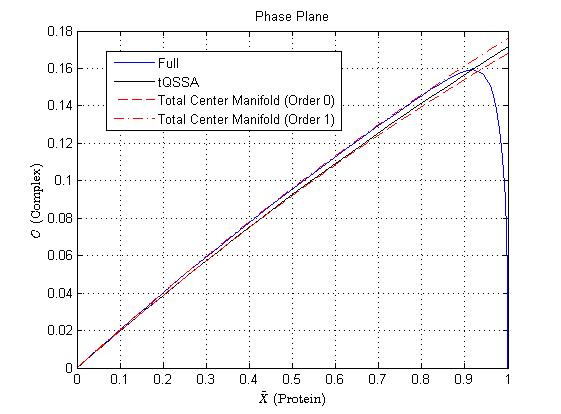}\hfil
\includegraphics[width=0.5\textwidth]{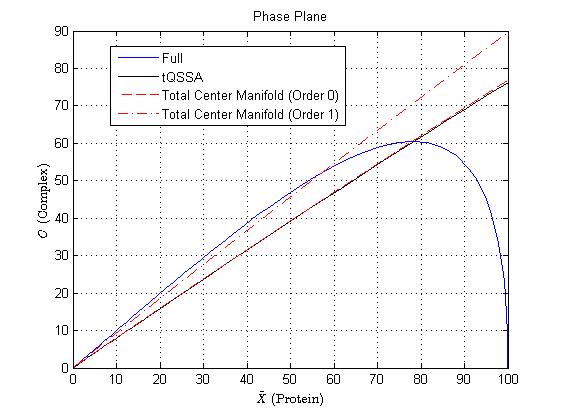}
\caption
{Comparison in the phase space $(\bar{X},C)$ of the numerical solution of the system (\ref{eq:aaa12}) (blue solid line) with its tQSSA (\ref{eq:a29}) (black solid line) and its zeroth order (dashed line) and first order (dashed/dotted line) center manifold (\ref{eq:aaa30}). The parameter sets are the following.
Left: $k_1=k_2=1; k_{-1}=3; E_T=1; X_T=1; K_M=4; K=1; \varepsilon_{HTA}=1; \varepsilon_{SS}=0.2; \varepsilon=0.03$. The set was taken from \cite{Ku11}. Right: $k_1=0.1; k_2=10; k_{-1}=0.01; E_T=400; X_T=100; K_M=100.1; K=100; \varepsilon_{HTA}=4; \varepsilon_{SS}=2; \varepsilon=0.11$. The set was taken from \cite{bulgari}.
In the plot on the left, since the value of $\varepsilon$ is sufficiently small, the different approximations converge to the graph of the numerical solution. In the plot on the right it is possible to appreciate the different behavior of the zero-th order and the first order center manifolds. While the first order manifold approximates in a better way the numerical solution, the zero-th order converges to the tQSSA, that does not approximate sufficiently well the  numerical solution, since it is the zero-th order term of the singular perturbation of the solution in terms of the parameter $\varepsilon$.}
\label{figura3}
\end{figure}

\subsection{A more general viewpoint}
Let us consider now a more general system of the following form ({\it outer solution})

\begin{equation}
\label{eq:bb1}
\begin{cases}
\frac{d u}{d\tau}&= \varphi\left(u,v\right),  \\
\varepsilon \frac{d v}{d\tau}&=au+bv+\psi\left(u,v\right), \ \ a,b\in\mathbb R, \ b<0
\end{cases}
\end{equation}
and the corresponding {\it inner solution}
\begin{equation}
\label{eq:b1}
\begin{cases}
\frac{d u}{ds}&= \varepsilon \varphi\left(u,v\right),  \\
\frac{d v}{ds}&=au+bv+\psi\left(u,v\right), \ \ a,b\in\mathbb R, \ b<0
\end{cases}
\end{equation}
(with $\tau = \varepsilon s$) where
\begin{equation}
\label{eq:b2}
\varphi\left(0,0\right)=\psi\left(0,0\right)=0, \ \ \text{and} \ \ \psi_u\left(0,0\right)=\psi_v\left(0,0\right)=0.
\end{equation}

The origin is a fixed point for (\ref{eq:b1}). Heineken-Tsuchiya-Aris system (\ref{eq:a13}) and the system obtained by the tQSSA approximation (\ref{eq:a22}), are particular cases of the system (\ref{eq:b1})-(\ref{eq:b2}).

We are able to state a more general theorem concerning the center manifold, which is the main result of our paper.

Let $w:=au+bv$; hence:
$$\frac{d w}{ds}=a \frac{d u}{ds}+b\frac{d v}{ds}= a\left[\varepsilon \varphi\left(u,v\right)\right]+b\left[au+bv+\psi\left(u,v\right)\right]$$
$$=bw+a\varepsilon \varphi\left(u,\frac{w-au}{b}\right)+b \psi\left(u,\frac{w-au}{b}\right)$$
Doing so, system (\ref{eq:b1}) becomes, for $a,b\in\mathbb R$ and $b<0$,
\begin{align}
\label{eq:b3}
\frac{d u}{ds}&= \varepsilon \varphi\left(u,\frac{w-au}{b}\right),  \notag \\
\frac{d w}{ds}&=bw+a\varepsilon \varphi\left(u,\frac{w-au}{b}\right)+b \psi\left(u,\frac{w-au}{b}\right) \notag \\
\frac{d \varepsilon}{ds}&=0
\end{align}
The associated linearized system has a block form of type (\ref{eq:34}) and, in fact, the eigenvalues are given by $0$ (with multiplicity $2$) and $b<0$. Thus in every system of the form (\ref{eq:b3}) we are in presence of a center manifold.

We write equation (\ref{eq:40}) for system (\ref{eq:b3}), employing Theorem \ref{th:5}. Referring to (\ref{eq:40}) and (\ref{eq:34}), we have that $A=0$, $B=b$. Accordingly, we search for a function $w=h(u,\varepsilon)$ such that
$$D_u h(u,\varepsilon) \left(0+\varepsilon \varphi\left(u,\frac{w-au}{b}\right)\right)+$$
$$-b h(u,\varepsilon)-a\varepsilon \varphi\left(u,\frac{w-au}{b}\right)-b \psi\left(u,\frac{w-au}{b}\right)=0$$
from which, since $D_u h(u,\varepsilon) \varepsilon \varphi\left(u,\frac{w-au}{b}\right)$ is a function at least of third order in $\varepsilon$ and $u$, while we are interested in a second order expression of function $h(u,\varepsilon)$, we can neglect this term and focus on
\begin{equation}
\label{eq:b4}
b h(u,\varepsilon)+a\varepsilon \varphi\left(u,\frac{w-au}{b}\right)+b \psi\left(u,\frac{w-au}{b}\right)=0
\end{equation}

\begin{theorem}
The center manifold of (\ref{eq:b1}) and the isolated point of (\ref{eq:b1}) are asymptotically equivalent.
\end{theorem}
\begin{proof}
{\bf Step 1.}

Using Theorem \ref{th:5} we assume
\begin{equation}
\label{eq:b5}
h(u,\varepsilon)=\lambda_1 u^2+\lambda_2 u \varepsilon+ \lambda_3 \varepsilon^2+\dots
\end{equation}
and it is trivial to prove that $h(u,\varepsilon)$ satisfies (\ref{eq:b4}) for $\lambda_3=0$. Moreover, from (\ref{eq:b2}),
$$\psi\left(u,\frac{w-au}{b}\right)=\frac{1}{2}\left[\Theta(u^2, uw, w^2)\right]+\dots$$
where $\Theta(u^2, uw, w^2)$ contains the quadratic terms in $u$ and $w$.

Since the terms in $uw$ and $w^2$, with $w=h(u,\varepsilon)=\lambda_1 u^2+\lambda_2 u \varepsilon+\dots$, are at least of third order in $\varepsilon$ and $u$, we consider only term in $u^2$. Therefore,
$$\psi\left(u,\frac{w-au}{b}\right)=\frac{1}{2} \left[\psi_{uu}(0,0)-2\frac{a}{b} \psi_{u,v}(0,0)+\left(\frac{a}{b}\right)^2 \psi_{v,v}(0,0)\right]u^2+\dots$$
while for $\varphi$ it is sufficient to consider the first order expansion in $u$ because, otherwise, in (\ref{eq:b4}) we would have third order terms for $\varepsilon \varphi\left(u,\frac{w-au}{b}\right)$ in $\varepsilon$ and $u$. Thus,
$$\varphi\left(u,\frac{w-au}{b}\right)=\left[\varphi_u(0,0)-\frac{a}{b} \varphi_v(0,0)\right]u+\dots$$
where we recall that $v=\frac{w-au}{b}$.
Accordingly, equation (\ref{eq:b4}) becomes:
\begin{align}
\label{eq:b6}
&b \left(\lambda_1 u^2+\lambda_2 u \varepsilon+\dots \right)+a\varepsilon u\left[\varphi_u(0,0)-\frac{a}{b} \varphi_v(0,0)\right] \notag \\
&+\frac{b}{2} \left[\psi_{uu}(0,0)-2\frac{a}{b} \psi_{u,v}(0,0)+\left(\frac{a}{b}\right)^2 \psi_{v,v}(0,0)\right]u^2+\dots =0
\end{align}
Equating to zero terms of the same power gives
\begin{align}
\lambda_1&=-\frac{1}{2}\left[\psi_{uu}(0,0)-2\frac{a}{b} \psi_{u,v}(0,0)+\left(\frac{a}{b}\right)^2 \psi_{v,v}(0,0)\right] \notag \\
\lambda_2&=-\frac{a}{b}\left[\varphi_u(0,0)-\frac{a}{b} \varphi_v(0,0)\right] \notag \\
\lambda_3&=0
\end{align}
Hence, the center manifold for system (\ref{eq:b1}) is
\begin{align}
\label{eq:b7}
w = h(u,\varepsilon)=&-\frac{1}{2}\left[\psi_{uu}(0,0)-2\frac{a}{b} \psi_{u,v}(0,0)+\left(\frac{a}{b}\right)^2 \psi_{v,v}(0,0)\right] u^2 \notag \\
&-\frac{a}{b}\left[\varphi_u(0,0)-\frac{a}{b} \varphi_v(0,0)\right] u \varepsilon+\dots
\end{align}

Setting in the RHS $\varepsilon = 0$, we obtain the center manifold $w=h(u,0)$ of (\ref{eq:b1}).

{\bf Step 2. Singular Point Technique \cite{Wa02}}

On the other hand,
$$v=\frac{w-au}{b}=\frac{\lambda_1 u^2+\lambda_2 u \varepsilon+ \lambda_3 \varepsilon^2+\dots-au}{b}$$
Since, setting $\varepsilon=0$, we have that $v=\frac{w-au}{b}=\frac{\lambda_1 u^2-au}{b}$, equation (\ref{eq:b4}) becomes

\begin{equation}
\label{eq:wiggins}
\lambda_1 u^2+\psi\left(u,v\right)\Biggl|_{v=\frac{\lambda_1 u^2-au}{b}}=0
\end{equation}
which gives an identity up to $O(u^2)$, if we substitute $\lambda_1$ as above and if we operate a Taylor expansion around $(u,v)=(0,0)$.

The vector field reduced to the center manifold, from equation (\ref{eq:35}) of Theorem \ref{th:3}, is:
\begin{align}
\frac{d u}{ds}&= \varepsilon \varphi\left(u,\frac{h(u,\varepsilon)-au}{b}\right),  \notag \\
\frac{d \varepsilon}{ds}&=0
\end{align}
or, in terms of the original time scale,
\begin{align}
\label{eq:b8}
\dot u&= \varphi\left(u,\frac{h(u,\varepsilon)-au}{b}\right)   \notag \\
\dot\varepsilon&=0
\end{align}

Moreover:
$$\varphi\left(u,v\right)=\left(\frac{\partial\varphi}{\partial u}\ \underbrace{\frac{\partial u}{\partial u}}_{=1}+ \frac{\partial\varphi}{\partial v}\ \underbrace{\frac{\partial v}{\partial u}}_{=-a/b}\right)u+\left(\frac{\partial\varphi}{\partial u}\ \underbrace{\frac{\partial u}{\partial w}}_{=0}+ \frac{\partial\varphi}{\partial v}\ \underbrace{\frac{\partial v}{\partial w}}_{=1/b}\right)w+\dots$$
where $v=\frac{w-au}{b}$ and all the derivatives are calculated in $(0,0)$. Hence,
$$\varphi\left(u,v\right)=\left[\varphi_ u(0,0) -\frac{a}{b} \varphi_v(0,0) \right]u+w\frac{\varphi_v(0,0)}{b}+\dots$$
and, since $\lambda_2=-\frac{a}{b}\left[\varphi_u(0,0)-\frac{a}{b} \varphi_v(0,0)\right]$, we have:
$$\varphi\left(u,v\right)=-\frac{b}{a}\lambda_2 u+w\frac{\varphi_v(0,0)}{b}+\dots$$
for $v=\frac{w-au}{b}$. From (\ref{eq:b5}), the vector field reduced to the center manifold, in terms of the original time scale, near the origin, becomes:
\begin{align}
\label{eq:b9}
\dot u&= -\frac{b}{a}\lambda_2 u+\left(\lambda_1 u^2+\lambda_2 u \varepsilon\right)\frac{\varphi_v(0,0)}{b}+o\left(\varepsilon^2+u^2\right)   \notag \\
\dot\varepsilon&=0
\end{align}
for values of $\lambda_1$ and $\lambda_2$ as above.

Summarizing, we have obtained two relations:

\textbf{a) From the Center Manifold Theory}: considering (\ref{eq:b4}), for $w=h(u,\varepsilon)$, and setting $\varepsilon = 0$, we have:
\begin{equation}
\label{eq:center1}
b w+b \psi\left(u,\frac{w-au}{b}\right)=0
\end{equation}

\textbf{b) From Singular Perturbation Techniques}: by assumption B of section \ref{sec:singpert}, and since $g(u,w)=bw+a\varepsilon \varphi\left(u,\frac{w-au}{b}\right)+b \psi\left(u,\frac{w-au}{b}\right)$ in (\ref{eq:b3}), we have, putting $\varepsilon=0$:

\begin{equation}
\label{eq:center2}
g(u, \phi(u)) = 0  \quad \Rightarrow \quad b\phi(u)+b \psi\left(u,\frac{\phi(u)-au}{b}\right)=0
\end{equation}

Comparing (\ref{eq:center1}) and (\ref{eq:center2}), we observe a relation between $h(u, 0)$ and $\phi(u)$ but we cannot infer that $h(u,0)=\phi(u)$, due to the non-uniqueness of center manifold. However, in the above steps we have proven that
\begin{equation}
h(u,0)\sim\phi(u), \quad {\rm for} \quad u\rightarrow0
\end{equation}
Q.E.D.
\end{proof}

This theorem means that the center manifolds obtained by means of (\ref{eq:center1}) and (\ref{eq:center2}) are asymptotically equivalent. This allows us to interpret any QSSA, obtained imposing $\dot{C} = 0$, as a manifold which is asymptotically equivalent to the center manifold, as confirmed by equation (\ref{eq:b10}) for Heineken-Tsuchiya-Aris system, and by equations (\ref{eq:a30})-(\ref{eq:a31}) for the tQSSA.

This explains why, in order to achieve the center manifold of (\ref{eq:a9s}) and (\ref{eq:aaa12}), it is sufficient to consider - for $u\rightarrow0$ - the expression obtained equating to zero the second equation of these systems, (i.e. for $\varepsilon=0$). We recall that in many papers (see, for example, \cite{Hein67,dvorak,Khoo}, who refer to Tihonov's Theorem, and \cite{Ku11}, according to Singular Perturbation Theory), the center manifold is obtained equating to zero the right hand side of the equation of the form:
$$\varepsilon\frac{dy}{dt}=g(x,y) \ .$$

\section{Conclusions and Perspectives}

The quasi-steady state approximation has been a challenge for applied mathematicians, who had to explain the feasibility of an approximation which imposes to the complex $C$ both to be constant and to depend on $X$. Some Biochemistry texts (see, for example, \cite{lehninger,yeremin,price,hammes}) mislead the reader, interpreting the QSSA as a true equality, which brings to assert that the ratio $\displaystyle {E(t)\,S(t)}/{C(t)}$ is constant during all the quasi-steady state phase. This is obviously not true. In \cite{anything} the authors solve the apparent incongruence, determining the asymptotic value of $\displaystyle {E(t)\,S(t)}/{C(t)}$, showing that, for every choice of the kinetic parameters and of the initial conditions,

\begin{equation}
\label{KW} \frac{E_{as}\,X_{as}}{C_{as}}(t)
\rightarrow \left(\frac{k_{2}-\alpha}{\alpha}\right)E_T=:K_W \qquad ; \qquad \left( K_D < K_W < K_M \right)
\end{equation}
(where
$
\alpha = \frac{k_1}{2} (K_M+E_T)
\left[1-\sqrt{1-\frac{4k_{2}E_T}{k_{1}(K_M+E_T)^2}}\;\right]
$
), differently from what is wrongly stated.

Heineken et al. \cite{Hein67} and successively other authors \cite{SS,SchMai02,perturbation} interpreted sQSSA and tQSSA as leading order expansions of the solutions in terms of a suitable parameter, which has to be considered small.

This interpretation allows us to embed the QSSA theory in a framework which is related to Tihonov's Theorem \cite{Tik48,Tik50,Tik52,Hein67,SS,Wa02,dvorak,Khoo}, where the parameter multiplies the derivative of $C$ and the QSSA can be obtained as the singular point of the original system, setting $\varepsilon = 0$.

In this paper we have shown that, at least in the classical simple scheme (\ref{eq:a3}), the approximation obtained applying Tihonov's Theorem is asymptotically equivalent to the center manifold of the system, which means that reduced system and center manifold are two sides of the same coin.

Once again, the total QSSA has shown to be much more efficient and natural than the standard one, mainly thanks to the fact that the parameter used for the expansions in the total framework is always less than $\displaystyle \frac{1}{4}$.

In our actual researches we are applying the techniques shown in this paper to more complex enzyme reactions, as the fully competitive inhibition \cite{CAIM}, the phosphorylation-dephosphorylation cycle (or Goldbeter-Koshland switch \cite{goldkosh}), the linear double phosphorylation reaction, the double phosphorylation-dephosphorylation cycle \cite{wang2} and, more in general, futile cycles \cite{wang}.

These mechanisms were already studied in terms of tQSSA in previous papers \cite{us_bmb,fima,pbb2,us_febs,us_ultra,GDA_phospho,GDA_bistability,Giovanna}.

The techniques here shown will allow to read the tQSSA as the leading term of an asymptotic expansion in terms of a suitable perturbation parameter, in these more complex cases, too.

\section*{Acknowledgements}

The authors are deeply grateful to Prof. Enzo Orsingher, from Sapienza University (Rome, Italy) and Prof. Jan Andres, from Palacky University (Olomouc, Czech Republic) for their translations of the papers \cite{Tik48,Tik50,Tik52} and some precious clarifications concerning some passages of the papers.




\end{document}